\documentclass[11pt]{amsart}
\usepackage{amsmath, amsthm, amsfonts, amsbsy, amssymb, upref, hyperref, pdfsync, enumerate} 
\usepackage{epsfig, graphicx}
\usepackage{pdfsync, hyperref}

\newcommand{\I}{\mathrm{i}}

\newcommand{\ep}{\epsilon}

\newtheorem{thm}{Theorem}[section]
\newtheorem{lmm}[thm]{Lemma}

\newtheorem{prop}[thm]{Proposition}

\newcommand{\bn}{{\bf n}}

\newcommand{\ee}{\mathbb{E}}

\newcommand{\pp}{\mathbb{P}}

\newcommand{\ra}{\rightarrow}
\newcommand{\rr}{\mathbb{R}}

\newcommand{\var}{\mathrm{Var}}

\begin{document}
\title{Fluctuations of the Bose-Einstein condensate}
\author{Sourav Chatterjee}
\address{Courant Institute of Mathematical Sciences, New York University, 251 Mercer Street, New York, NY 10012}
\email{sourav@cims.nyu.edu}
\author{Persi Diaconis}
\address{Department of Statistics, Stanford University, 390 Serra Mall, Stanford, CA 94305}
\email{diaconis@math.stanford.edu}
\thanks{Sourav Chatterjee's research was partially supported by the  NSF grant DMS-1005312}
\thanks{Persi Diaconis's research was partially supported by NSF grant DMS-0804324}

\begin{abstract}
This article gives a rigorous analysis of the fluctuations of the Bose-Einstein condensate for a system of non-interacting bosons in an arbitrary potential, assuming that the system is governed by the canonical ensemble. As a result of the analysis, we are able to tell the order of fluctuations of the condensate fraction as well as its limiting distribution upon proper centering and scaling. This yields interesting results. For example, for a system of $n$ bosons in a 3D harmonic trap near the transition temperature, the order of fluctuations of the condensate fraction is $n^{-1/2}$ and the limiting distribution is normal, whereas for the 3D uniform Bose gas, the order of fluctuations is $n^{-1/3}$ and the limiting distribution is an explicit non-normal distribution. For a 2D harmonic trap, the order of fluctuations is $n^{-1/2}(\log n)^{1/2}$, which is larger than $n^{-1/2}$ but the limiting distribution is still normal. All of these results come as easy consequences of a general theorem. 
\end{abstract}
\maketitle

\section{Introduction}
Consider a system of $n$ non-interacting particles, each of which can be in one of a discrete set of quantum states. If the particles are distinguishable, then the state of the system is described by the $n$-tuple consisting of the states of the $n$  particles. On the other hand if the particles are {\it indistinguishable}, then the state of the system is described by the sequence $(n_0, n_1, n_2,\ldots)$, where $n_j$ is the number of particles in state $j$. From now on we will only consider indistinguishable particles (bosons). If state $j$ has energy $E_j$, then the total energy of the system is $\sum_{j=0}^\infty n_j E_j$. (Here `energy' means the energy associated with a given state; that is, if the state is an eigenfunction of a Schr\"odinger operator, then the energy is the corresponding eigenvalue.) 

If the system is in thermal equilibrium, then the Boltzmann hypothesis implies that the chance of observing the system in state $(n_0, n_1,\ldots)$ is proportional to 
\[
\exp\biggl(-\frac{1}{k_BT} \sum_{j=0}^\infty n_j E_j\biggr),
\]
where $T$ is the temperature and $k_B$ is Boltzmann's constant. This is the so-called `canonical ensemble' for a system of $n$ non-interacting bosons with energy levels $E_0, E_1, \ldots$. Typically, the $E_j$'s are arranged in increasing order, so that $E_0$ is the energy of the ground state. 

Building on the work of Bose \cite{bose}, Einstein \cite{einstein24, einstein25} realized that for a system of indistinguishable particles modeled as above, there is a transition temperature below which a macroscopic fraction of the particles settle down in the ground state. This phenomenon is known as Bose-Einstein condensation. The first realization of a Bose-Einstein condensate was obtained in 1995 \cite{andersonetal, davisetal}, resulting in an explosion of activity in this field. For modern treatments of the subject and surveys of the physics literature, see e.g.~\cite{dalfovo, leggett, kocharov, kocharovsky, pethicksmith}. For rigorous mathematical results and further references, see~\cite{aizenman, beau, erdosetal06, erdosetal10, liebseiringer, lieb, tamura, verbeure11}. 

Besides the canonical model described above, there are two other standard approaches to modeling Bose-Einstein condensation. The {\it grand canonical ensemble} assumes that the system is allowed to exchange particles with a neighboring particle reservoir. In other words, the total number  of particles is  allowed to vary. In the grand canonical ensemble, the possible states of the system are all sequences $(n_0,n_1,\ldots)$ of non-negative integers  (instead of sequences summing to a fixed $n$), and the probability that the system is in state $(n_0, n_1,\ldots)$ is proportional to
\[
\exp\biggl(-\frac{1}{k_BT} \sum_{j=0}^\infty n_j (E_j- \mu)\biggr),
\]
where $\mu$ is a quantity called the `chemical potential'. Note that $\mu$ must be strictly less than $E_0$ for this to be a proper probability measure. Given a temperature $T$ and an expected particle number $n$, the chemical potential $\mu$ is determined using the condition that the expected number of particles equals $n$ at temperature $T$. 

The {\it microcanonical ensemble}, on the other hand, assumes that both the total number of particles as well as the total energy are given (or given approximately), and the state of the system is drawn uniformly at random from all possible states satisfying these two constraints.

It is a general characteristic of physical systems that the three ensembles have similar behavior (``equivalence of ensembles''). However, the physics community realized --- surprisingly recently, about twenty years ago --- that  Bose gases  exhibit a striking departure from this general rule, in the regime where the condensate appears: The fluctuations of the size of the condensate fraction are of macroscopic size in the grand canonical ensemble, whereas they are microscopic in the other two ensembles. The problem of condensate fluctuations in the context of modern experiments was brought to the attention of the theoretical physics community by the paper of Grossmann and Holthaus \cite{gh}. Then, the idea of H.~D.~Politzer of treating the condensate as a reservoir of particles for the excited subsystem was used to define the concept of the ``fourth statistical ensemble'' in \cite{navez}.  The formalism defined there soon become a major tool for the determination of fluctuations in the canonical and microcanonical ensembles for a large class of potentials \cite{gh2, ww}. For detailed surveys of the above line of work and further references, see e.g.~\cite{beau,  buffet, fannes, gajda, holthaus, jaeck, leggett, lewis, kocharovsky,  politzer96,  pule, ziff}. 

In the next section, we present a theorem that gives a mathematically rigorous solution to the problem of understanding the fluctuations of the condensate fraction in the canonical ensemble for trapped non-interacting Bose gases. The main differences between the large body of literature cited above and our work are that (a) we have theorems with rigorous proofs, and (b) our theorems work for more general potentials than the ones considered in prior work. An extension of our work to the microcanonical ensemble is the topic of a future paper, soon to be finished. 

On the rigorous side, there is a body of work relating fine properties of Bose-Einstein condensation with those of random partitions, possibly beginning with  a paper of Vershik \cite{vershik}. This approach is particularly suitable for studying Bose-Einstein condensation under mean-field interactions. Some of the notable papers in this direction are \cite{adams, dvz, suto, bcmp}. Among other things, there are some rigorous results about fluctuations of the condensate in these papers.  This is a more challenging and ambitious direction than what we consider in this paper since it involves interactions.  However the non-interacting system is still important for certain conceptual purposes. See \cite{politzer96} and  \cite{dalfovo} for details. An accessible introduction to the ideal Bose gas and its condensation is in Krauth \cite[Chapter 4]{krauth}.

\section{Results}
Consider a system of $n$ non-interacting bosons of mass $m$ in a potential~$V$. Suppose that the potential $V$ is such that the Schr\"odinger operator 
\begin{equation}\label{schro}
\hat{H} = -\frac{\hbar^2}{2m}\Delta + V
\end{equation}
has a discrete spectrum $E_0< E_1\le E_2\le\cdots$. These are the possible values of the energy of a single particle. Note that we have assumed that the ground state energy $E_0$ is strictly less than $E_1$, but the other inequalities are not strict. The strict inequality $E_0 < E_1$ is another way of saying that the ground state is unique. 

Incidentally, if the energy levels are allowed to depend on $n$, and in particular if the gap $E_1-E_0$ decays to zero as $n$ goes to infinity, the situation becomes mathematically more challenging. See Lieb et al.~\cite{lieb} for details. But we will not worry about this scenario here.


A {\it configuration} describing the state of the system is a sequence of the form $\bn = (n_0, n_1,\ldots)$ where $n_0, n_1,\ldots$ are non-negative integers summing to $n$. Here $n_j$ stands for the number of particles in energy eigenstate $j$. Then the energy of a configuration $\bn$ is 
\[
H(\bn) = \sum_{j=0}^\infty n_j E_j. 
\]
The {\it canonical Gibbs measure} is the probability measure on the space of configurations that puts mass proportional to $e^{-\beta H(\bn)}$ at $\bn$. Here 
\[
\beta = \frac{1}{k_BT}\ , 
\]
where $T$ is the temperature and $k_B$ is Boltzmann's constant. Call $\beta$ the `inverse temperature'.

Suppose that there exist constants $L > 0$ and $\alpha\ge  1$  such that 
\begin{equation}\label{weyl}
\lim_{\lambda \ra \infty} \frac{\#\{j : E_j \le \lambda\}}{L\lambda^\alpha} = 1. 
\end{equation}
It will follow from the theorems stated below that the numbers $L$ and $\alpha$ are sufficient to determine the size of the condensate fraction as well as its fluctuations in the limit as $n \ra\infty$. No other feature of the energy spectrum is relevant. If $\alpha > 1$ define
\[
t_c := \frac{1}{k_B(L\alpha \Gamma(\alpha)\zeta(\alpha))^{1/\alpha}}
\]
where $\Gamma$ and $\zeta$ are the classical Gamma and Zeta functions. If $\alpha = 1$ define
\[
t_c := \frac{1}{k_BL}. 
\]
Fix any $t > 0$. 
Let $(N_0, N_1,\ldots)$ be a random configuration drawn from the canonical Gibbs measure for a system of $n$ bosons at the temperature 
\[
T = 
\begin{cases}
tn^{1/\alpha} & \text{ if } \alpha > 1,\\
\frac{tn}{\log n} &\text{ if } \alpha = 1.
\end{cases}
\]
The number $T_c$ is defined as the value of $T$ when $t=t_c$. This will be called the `transition temperature' for the system of $n$ particles. Note that $T/T_c = t/t_c$. Our first theorem gives the limiting value of the condensate fraction at temperature $T$ when $T$ is below the transition temperature.
\begin{thm}\label{wlln}
Suppose that $t< t_c$. Then as $n \ra\infty$, $N_0/n$ converges in probability to $1-(t/t_c)^\alpha$. 
\end{thm}
Theorem \ref{wlln} shows that indeed, the  `critical' or `transition' temperature for the appearance of the condensate is  $T_c$, since the condensate fraction tends to zero as $T$ approaches $T_c$. While this formula for the size of the condensate fraction is well known for specific potentials (see e.g.~\cite{pethicksmith}), the general formula for arbitrary potentials  in terms of the coefficients $L$ and $\alpha$ is possibly a new result.

Figure \ref{frac} illustrates the graphs of the limiting condensate fraction versus the temperature for three different values of $\alpha$, corresponding to three interesting potentials. 

\begin{figure}[t!]
\centering
\includegraphics[scale=0.3, trim=0in 0in 0in 0in, clip=true]{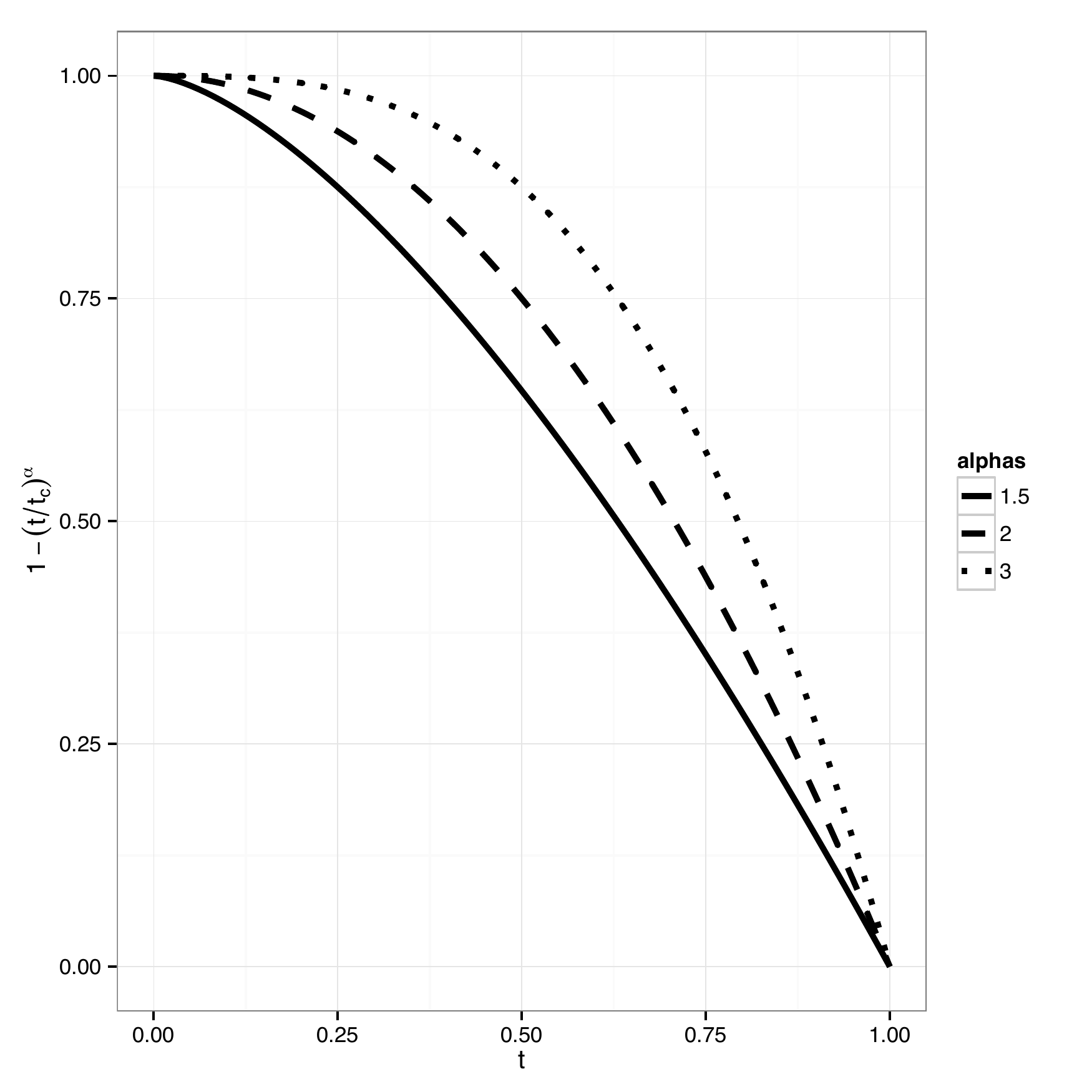}
\caption{Graphs of the condensate fraction versus the temperature for $\alpha=1.5$, $2$ and $3$. The three cases correspond to the uniform Bose gas in a box, the Bose gas in a 2D harmonic trap, and the Bose gas in a 3D harmonic trap, respectively. Units are adjusted such that $t_c=1$ in all three cases. (Courtesy: Susan Holmes.)}
\label{frac}
\end{figure}

Our second theorem gives the limiting distribution of the condensate fraction upon proper centering and scaling, in the temperature regime where the condensate appears. This is the main result of this paper. In the following, $\stackrel{d}{\longrightarrow}$ will stand for convergence in distribution and ${\mathcal N}(\mu,\sigma^2)$ will denote the normal distribution with mean $\mu$ and variance $\sigma^2$. Let $X_1, X_2,\ldots$ be i.i.d.\ exponential random variables with mean $1$, and define 
\begin{equation}\label{vdef}
W:= \frac{1}{(L\alpha \Gamma(\alpha)\zeta(\alpha))^{1/\alpha}}\sum_{j=1}^\infty \frac{1-X_j}{E_j-E_0}\  \ \text{ if } \ \alpha >1,
\end{equation}
and 
\begin{equation}\label{vdef2}
W:= \frac{1}{L}\sum_{j=1}^\infty \frac{1-X_j}{E_j-E_0}\  \ \text{ if } \ \alpha =1,
\end{equation}
provided that the infinite series converges almost surely. If the series does not converge almost surely, then $W$ is undefined. 
\begin{thm}\label{clt}
Suppose that $t< t_c$. If $\alpha = 1$, then the infinite series in \eqref{vdef2} converges almost surely and in $L^2$, and as $n \ra \infty$, 
\[
\frac{N_0-\ee(N_0)}{n/\log n} \stackrel{d}{\longrightarrow} (t/t_c) W. 
\]
If $1< \alpha <2$, then the infinite series in \eqref{vdef} converges almost surely and in $L^2$, and as $n \ra \infty$, 
\[
\frac{N_0-\ee(N_0)}{n^{1/\alpha}} \stackrel{d}{\longrightarrow} (t/t_c) W. 
\]
If $\alpha = 2$, then 
\[
\frac{N_0-\ee(N_0)}{\sqrt{n\log n}} \stackrel{d}{\longrightarrow} {\mathcal N}(0, 3(t/t_c)^2/\pi^2).
\]
If $\alpha > 2$, then
\[
\frac{N_0-\ee(N_0)}{\sqrt{n}} \stackrel{d}{\longrightarrow} {\mathcal N}(0, (t/t_c)^\alpha \zeta(\alpha-1)/\zeta(\alpha)).
\]
\end{thm}
In spite of the wealth of literature on the non-interacting Bose gas, the above theorem appears to be a genuinely new result. As applications of Theorem \ref{clt}, consider several special cases which demonstrate the four situations of Theorem \ref{clt}. All are real, interesting examples considered in the physics literature. 
\begin{enumerate}[(i)]
\item {\bf 3D harmonic trap.} The potential function is 
\[
V(x,y,z)= \frac{m\omega^2}{2}(x^2+y^2+z^2) 
\]
where $\omega$ is some positive constant. 
It is well known (see e.g.\ \cite{dalfovo}) that  the eigenvalues of the Schr\"odinger operator \eqref{schro} for this potential are  exactly
\[
\biggl(i+j+k - \frac{3}{2}\biggr)\hbar \omega,
\]
where $i,j,k$ range over all non-negative integers. From this it is easy to see that $L= 1/(6\hbar^3\omega^3)$ and $\alpha = 3$ here. Consequently, the fluctuations of the condensate fraction are of order $n^{-1/2}$ and have a limiting normal distribution upon suitable centering and scaling. 

\item {\bf 3D cubical box.} There is a $K$ such that $V(x,y,z) = 0$ if $(x,y,z)\in [0,K]^3$ and $=\infty$ otherwise. It is not difficult to verify that the energy levels of the Schr\"odinger operator for this potential are of the form $Ci^2+Cj^2+Ck^2$ as $i,j,k$ range over non-negative integers, where $C$ is some constant depending on $K$ and $\hbar$. This shows that $L=4\pi/3C^{3/2}$ and  $\alpha = 3/2$ for the uniform gas and therefore, by Theorem \ref{clt},  the fluctuations of the condensate fraction are of order $n^{-1/3}$ and the limiting distribution is non-normal. A histogram of 10,000 simulated values of $W$ for this potential is given in Figure~\ref{histo}. 

\begin{figure}[t!]
\centering
\includegraphics[scale=0.2, trim=0in 0in 0in 0in, clip=true]{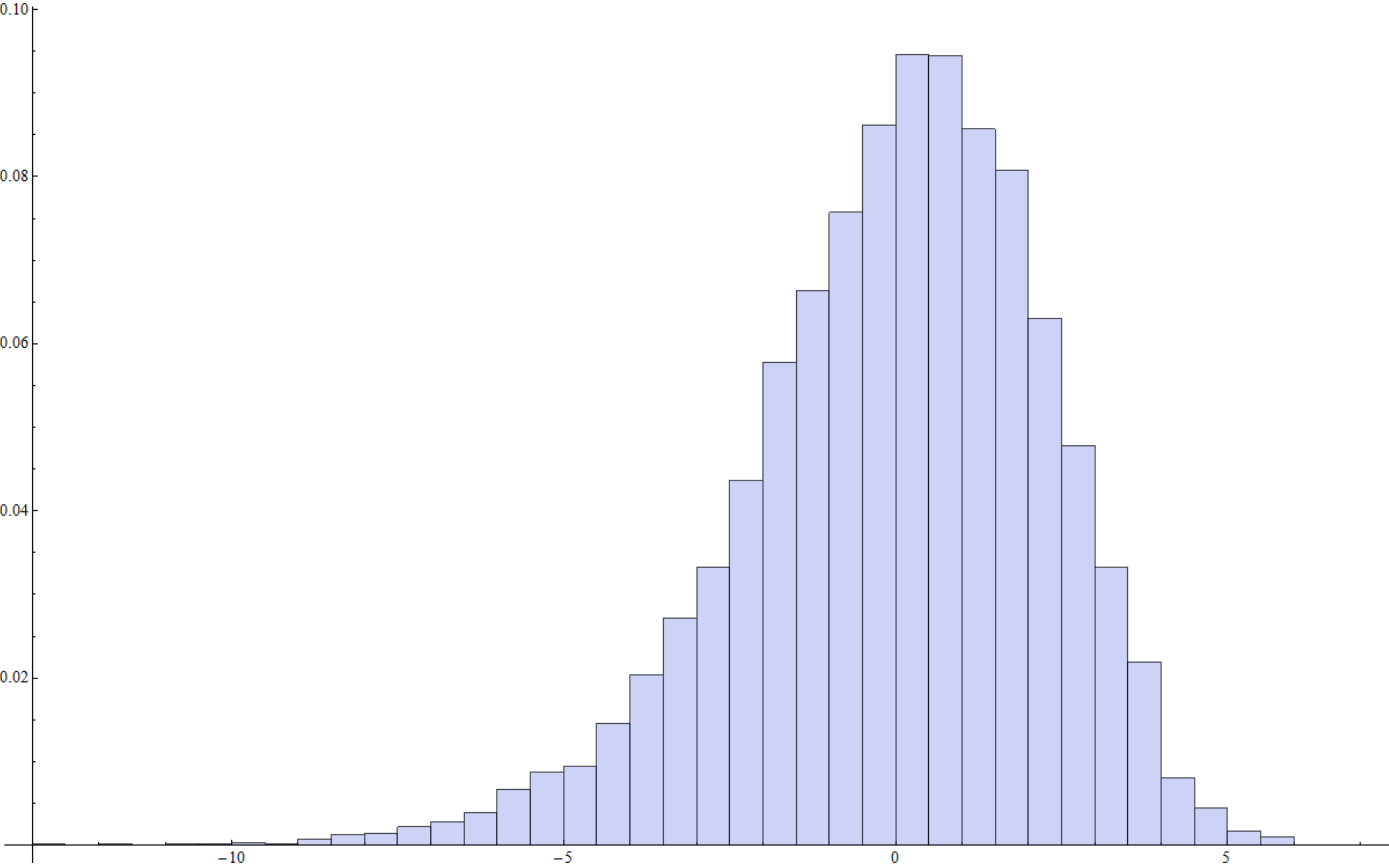}
\caption{Histogram of 10,000 simulated values of $W$ for the uniform Bose gas with energy levels $i^2+j^2+k^2$, with $i$, $j$ and $k$ ranging over non-negative integers. (Courtesy: James Zhao)}
\label{histo}
\end{figure}

\item {\bf 2D harmonic trap.} Just as in the 3D case, the eigenvalues are of the form $Ci+Cj$ as $i,j$ range over non-negative integers and $C$ is some constant. Clearly, $L=1/2C^2$ and $\alpha=2$ for this potential, and therefore the fluctuations of the condensate fraction are of order $n^{-1/2}(\log n)^{1/2}$ and  normally distributed in the limit. 

\item {\bf 1D harmonic trap.} The eigenvalues are of the form $Cj$ as $j$ ranges over non-negative integers and $C$ is some constant. Clearly, here $\alpha = 1$ and $L=C$. Therefore the fluctuations of the condensate fraction are of order $1/\log n$ and the limiting distribution is non-normal. 

\item {\bf 2D square box.} The eigenvalues are of the form $Ci^2+Cj^2$ where $i$ and $j$ range over non-negative integers are $C$ is some constant. In this case, an easy computation gives $\alpha =1$ and $L=\pi/4C$. The fluctuations of the condensate fraction are of order $1/\log n$ and the limiting distribution is non-normal. 

\end{enumerate}
For a general potential, the constants $L$ and $\alpha$ may often be obtained with the aid of Weyl's law for eigenvalues of Schr\"odinger operators, without going through the trouble of actually diagonalizing the operators. For details, see e.g.~\cite[Chapter 6]{zworski}. 

Our next theorem  gives the limiting distributions of the other cell counts, in the temperature regime where the condensate appears. For $N_1, N_2,\ldots$, the fluctuations are of the same size as their expected values; in other words, these cell counts are not concentrated. This result is a byproduct of the proof of Theorem \ref{clt}. 
\begin{thm}\label{other}
Suppose that $t< t_c$ and $\alpha >1$. Then for any fixed $j\ge 1$, as $n$ tends to infinity, the limiting distribution of $N_j/n^{1/\alpha}$ is exponential with mean $k_B t/(E_j-E_0)$. When $\alpha=1$, the same result holds if $n^{1/\alpha}$ is changed to $n/\log n$. 
\end{thm} 
It is natural to ask questions about the behavior of the random variable $W$ in Theorem \ref{clt}. Figure \ref{histo} already gives an indication that the distribution of $W$ is non-symmetric around zero. The following theorem shows that the left and right tails of $W$ indeed behave differently. In particular, it shows that if the $E_j$'s satisfy \eqref{weyl} with $\alpha >1$, then the right tail of $W$ falls off as $\exp(-cx^{1+1/(\alpha-1)})$ and the left tail falls off exponentially. For example, for the 3D uniform Bose gas in a box (depicted in Figure \ref{histo}), $\alpha = 3/2$ and so the upper tail of $W$ behaves as~$\exp(-cx^3)$. When $\alpha=1$, the right tail falls off double exponentially. 
\begin{thm}\label{wtail}
Let $W$ be defined as in \eqref{vdef} and assume, for simplicity, that the constant in front of the sum is $1$. Suppose that $\sum_{j=1}^\infty 1/(E_j-E_0)^2< \infty$. Given $x\in (0,\infty)$, choose $n_x$ so that $\sum_{j\le n_x} 1/(E_j-E_0) \le x/2$. Then 
\begin{equation}\label{wx1}
\pp(W\ge x) \le \exp\biggl(-\frac{x^2}{8\sum_{j> n_x} 1/(E_j-E_0)^2}\biggr). 
\end{equation}
If $n_x'$ is such that $\sum_{j\le n_x'} 1/(E_j-E_0) \ge 2x$, then 
\begin{equation}\label{wx2}
\pp(W\ge x) \ge 2^{-22} \exp\biggl(-\frac{120 x^2}{\sum_{j> n_x'} 1/(E_j-E_0)^2}\biggr). 
\end{equation}
In particular, if the $E_j$'s satisfy \eqref{weyl}, then there are constants $a_1$, $a_2$, $a_3$, $a_4$, $a_5$ and $a_6$ depending only on the $E_j$'s such that for all $x> 0$,
\begin{align}
a_1 e^{-a_2 x^{1+1/(\alpha-1)}} &\le \pp(W\ge x) \le a_3 e^{-a_4 x^{1+1/(\alpha-1)}} \ \ \text{ if }  \ 1<\alpha < 2, \ \text{ and} \label{a1a2}\\
a_1 e^{-a_2e^{a_3x}} &\le \pp(W\ge x) \le a_4 e^{-a_5 e^{a_6x}} \ \ \text{ if } \  \alpha=1. \label{a1a22}
\end{align}
Lastly, there are constants $c_1, c_2 > 0$ such that for all $x> 0$,
\begin{equation}\label{wx3}
\pp(W\le -x) \le c_1 e^{-c_2 x},
\end{equation}
and the right hand side cannot be improved to $c_1 e^{-c_2 x^{1+\ep}}$ for any $\ep >0$. 
\end{thm}
Interestingly, there is one special case where the distribution of $W$ can be computed explicitly. This is the case of the one-dimensional harmonic trap, which has energy levels $0,1,2,3,\ldots$. This is closely connected with the behavior of random partitions of integers \cite{erdos41} and extreme value theory. 
\begin{prop}\label{gumbel}
If $E_j=j$ for $j=0,1,2,\ldots$, and $W$ is defined as in \eqref{vdef2} (with $L=1$), then for all $x\in \rr$
\[
\pp(W\ge x) = e^{-e^{x-\gamma}},
\]
where $\gamma$ is Euler's constant. 
\end{prop}
Our final theorem is a law of large numbers that records the correspondence between the total energy of the system and the temperature.
\begin{thm}\label{ellnthm}
Let $E_{tot} := \sum_{j=0}^\infty N_j E_j$ be the total energy of the system at temperature $T$, where $T=tn^{1/\alpha}$ if $\alpha >1$ and $T=tn/\log n$ if $\alpha =1$. Here $t$ is a fixed constant, strictly less than $t_c$. If $\alpha > 1$, then 
\[
\frac{E_{tot}}{T^{1+\alpha}} \ra k_B^{1+\alpha} L\alpha \Gamma(\alpha+1)\zeta(\alpha+1) \ \ \text{ in probability as $n\ra\infty$}. 
\]
If $\alpha = 1$, then 
\[
\frac{E_{tot}}{T^2} \ra \frac{k_B^{2} L \pi^2}{6} \ \ \text{ in probability as $n\ra\infty$}. 
\]
\end{thm}
The next section contains the proofs of the theorems presented in this section.

\section{Proofs}
Note that the canonical Gibbs measure and the distribution of $W$ remain the same if we add or subtract a fixed constant from each $E_j$. The constants $L$ and $\alpha$ are also invariant under such a transformation, as are the limiting total energies in Theorem \ref{ellnthm}. Therefore, from now on, we will assume without loss of generality that $E_0=0$.

Take any continuous function $\phi:[0,\infty) \ra\rr $ such that 
\begin{equation}\label{mk}
\sum_{k=0}^\infty m_k (k+1)^{\alpha} < \infty,
\end{equation}
where
\[
m_k := \max_{k\le x\le k+1} |\phi(x)|. 
\]
Note that under the above condition,   
\begin{equation}\label{phint}
\int_0^\infty |\phi(x)| x^{\alpha -1} dx < \infty. 
\end{equation}
\begin{lmm}\label{betalmm}
For a function $\phi$ as above, 
\[
\lim_{\beta \downarrow 0} \beta^{\alpha} \sum_{j=1}^\infty \phi(\beta E_j) = L\alpha \int_0^\infty \phi(x) x^{\alpha-1} dx. 
\]
\end{lmm}
\begin{proof}
For each $\lambda > 0$ let 
\begin{equation}\label{sdef}
S(\lambda) := \#\{j : E_j \le \lambda\}. 
\end{equation}
By \eqref{weyl}, there exists $C > 0$ such that for all $\lambda > 0$, 
\begin{equation}\label{sbd}
S(\lambda) \le C\lambda^{\alpha}. 
\end{equation}
Take any $\ep > 0$. Using \eqref{mk} and \eqref{phint}, choose an integer $K$ so large that 
\begin{equation}\label{mchoice}
C\sum_{k\ge K} m_k (k+1)^\alpha < \ep, \ \text{ and } \ L\alpha \int_K^\infty |\phi(x)| x^{\alpha -1} dx <\ep. 
\end{equation}
For each $\beta \in (0, K/E_1)$, define the probability measure
\[
\mu_\beta := \frac{1}{S(K/\beta)} \sum_{j\, :\, E_j \le K/\beta}\delta_{\beta E_j},
\]
where $\delta_x$ denotes the point mass at $x$. 
Then 
\begin{align*}
\sum_{j=1}^\infty \phi(\beta E_j) &= \sum_{j\, : \, \beta E_j \le K} \phi(\beta E_j) + \sum_{j\, :\, \beta E_j > K } \phi(\beta E_j) \\
&= S(K/\beta) \int_0^\infty \phi(x) d\mu_\beta(x)  + \sum_{j\, :\, \beta E_j > K } \phi(\beta E_j). 
\end{align*} 
Now note that by \eqref{weyl}, for any $x\in (0, K)$,
\begin{align*}
\lim_{\beta \downarrow 0} \mu_\beta([0, x]) &= \lim_{\beta \downarrow 0} \frac{\#\{j: \beta E_j \le x\}}{S(K/\beta)}\\
&= \lim_{\beta \downarrow 0} \frac{S(x/\beta)}{S(K/\beta)} = \frac{x^\alpha}{K^\alpha}. 
\end{align*}
This shows that as $\beta\ra 0$, $\mu_\beta$ tends weakly to the probability measure on $[0,K]$ with probability density function $\alpha x^{\alpha -1} K^{-\alpha}$. Since $\phi$ is continuous on $[0,\infty)$ (and therefore bounded and continuous on $[0,K]$), this shows that
\[
\lim_{\beta \ra 0} \int_0^\infty \phi(x) d\mu_\beta (x) = \frac{\alpha}{K^\alpha}\int_0^K \phi(x) x^{\alpha-1} dx. 
\]
Next, observe that by \eqref{sbd} and \eqref{mchoice}, 
\begin{align*}
\biggl|\beta^\alpha \sum_{j\, :\, \beta E_j > K } \phi(\beta E_j)\biggr| &\le \beta^\alpha\sum_{k\ge K} m_k \#\{j: k \le \beta E_j< k+1\}\\
&\le \beta^\alpha\sum_{k\ge K} m_k S((k+1)/\beta) \\
&\le C\sum_{k\ge K} m_k (k+1)^\alpha \le \ep.  
\end{align*}
Combining, we get
\begin{align*}
&\limsup_{\beta \ra 0} \biggl|\beta^\alpha \sum_{j=1}^\infty \phi(\beta E_j) - L\alpha \int_0^\infty \phi(x)x^{\alpha -1} dx \biggr| \\
&\le L\alpha \int_K^\infty |\phi(x)|x^{\alpha-1} dx + \limsup_{\beta \ra 0} \biggl|\beta^\alpha \sum_{j\, :\, \beta E_j > K } \phi(\beta E_j)\biggr|\\
&\le 2\ep. 
\end{align*}
Since this holds for any $\ep > 0$, the proof is complete. 
\end{proof}
Fix $\beta > 0$. Let $Z_1, Z_2,\ldots$ be independent random variables, with 
\[
\pp(Z_j = k) = e^{-\beta E_j k} (1-e^{-\beta E_j}), \ \ k=0,1,2,\ldots.
\]
That is, $Z_j$ has a geometric distribution with parameter $1-e^{-\beta E_j}$. Note that by \eqref{sbd}, and the facts that  $S(E_j) \ge j$ for all $j$ and $E_1>0$, it follows that there exists $K> 0$ such that 
\begin{equation}\label{ejbd}
E_j \ge K j^{1/\alpha}
\end{equation}
for all $j$. From this it is easy to conclude that the sequence $Z_1, Z_2,\ldots$ is summable almost surely.  Define 
\[
M := \sum_{j=1}^\infty Z_j.
\]
\begin{lmm}\label{wllnlemma}
Suppose that $\alpha > 1$. As $\beta \ra 0$, $\beta^\alpha M$ converges in probability to $L\alpha \Gamma(\alpha)\zeta(\alpha)$. 
\end{lmm}
\begin{proof}
Note that 
\begin{align}\label{eem}
\ee(M) &= \sum_{j=1}^\infty \frac{1}{e^{\beta E_j } -1}
\end{align}
and 
\begin{align}\label{varm}
\var(M) = \sum_{j=1}^\infty \frac{e^{-\beta E_j}}{(1-e^{-\beta E_j })^2}. 
\end{align}
By Lemma \ref{betalmm} and the above formulas, we get
\begin{align*}
\lim_{\beta\ra 0} \beta^\alpha \ee(M) &= L\alpha\int_0^\infty \frac{x^{\alpha -1} }{e^x -1} dx\\
&= L\alpha \sum_{k=1}^\infty \int_0^\infty x^{\alpha-1} e^{-kx} dx\\
&= L\alpha\Gamma(\alpha) \sum_{k=1}^\infty k^{-\alpha} = L\alpha \Gamma(\alpha) \zeta(\alpha),
\end{align*}
and 
\begin{align*}
\limsup_{\beta \ra 0} \beta^{1-\alpha}\var(\beta^{\alpha} M) &\le \limsup_{\beta \ra 0} \frac{\beta^{\alpha}}{E_1} \sum_{j=1}^\infty \frac{\beta E_j e^{-\beta E_j}}{(1-e^{-\beta E_j})^2} \\
&= \frac{L\alpha}{E_1} \int_0^\infty \frac{x^{\alpha} e^{-x}}{(1-e^{-x})^2}dx < \infty. 
\end{align*}
Since $\alpha>1$, this show that $\var(\beta^{\alpha} M) \ra 0$ as $\beta \ra 0$. This completes the proof of the lemma. 
\end{proof}
\begin{lmm}\label{wlln1}
Suppose that $\alpha =1$. Then as $\beta \ra 0$, $\beta M/\log (1/\beta)$ converges in probability to $L$. 
\end{lmm}
\begin{proof}
The identity \eqref{eem} may be rewritten (with the help of \eqref{ejbd}) as 
\begin{align*}
\ee(M) &= \sum_{j=1}^\infty \frac{e^{-\beta E_j}}{1-e^{-\beta E_j}}\\
&= \sum_{j=1}^\infty  \sum_{k=1}^\infty e^{-k\beta E_j} = \sum_{k=1}^\infty c_k(\beta), 
\end{align*}
where 
\[
c_k(\beta) := \sum_{j=1}^\infty e^{-k\beta E_j}. 
\]
By \eqref{ejbd}, there is a constant $K>0$ such that 
\[
c_k(\beta) \le \sum_{j=1}^\infty e^{-Kk\beta j}= \frac{1}{e^{Kk\beta} -1}. 
\]
Now, if $k\beta \ge 1/\log\log(1/\beta)$, then 
\begin{align*}
1 &= e^{Kk\beta-Kk\beta} \le e^{Kk\beta}e^{-K/\log\log(1/\beta)},
\end{align*}
and consequently,
\[
\frac{1}{e^{Kk\beta}-1} \le \frac{e^{-Kk\beta} }{1-e^{-K/\log\log(1/\beta)}}. 
\]
Thus, if 
\[
C(\beta):= \frac{1}{\beta \log\log (1/\beta)},
\]
then 
\begin{align}
\sum_{k\ge C(\beta)} c_k(\beta) &\le \frac{1}{(e^{K/\log\log(1/\beta)} -1) (1-e^{-K\beta})}\label{cupper}\\
&\sim \frac{\log\log(1/\beta)}{K^2\beta} \ \ \text{as } \ \beta \ra 0. \nonumber 
\end{align}
On the other hand, we claim that 
\begin{align}\label{kbeta}
\lim_{\beta \ra0}\max_{k\le C(\beta)} |k\beta c_k(\beta)-L| = 0. 
\end{align}
To prove this, suppose not. Then there exist sequences $\beta_i \ra 0$ and $k_i\le C(\beta_i)$ such that 
\[
k_i \beta_i c_{k_i}(\beta_i) \not \to L. 
\]
But since $k_i \beta_i \ra 0$, it follows as a simple consequence of Lemma \ref{betalmm} that the left hand side of the above display must tend to $L$ as $i\to \infty$. This gives a contradiction, proving \eqref{kbeta}. Applying \eqref{kbeta} gives 
\begin{align*}
\biggl|\sum_{k\le C(\beta)} c_k(\beta) - \sum_{k\le C(\beta)} \frac{L}{k\beta}\biggr| &\le \max_{k\le C(\beta)} |k\beta c_k(\beta)-L| \sum_{k\le C(\beta)} \frac{1}{k\beta} \\
&= o\biggl(\frac{\log(1/\beta)}{\beta}\biggr) \ \ \text{ as } \ \beta \ra0. 
\end{align*}
Combining this with \eqref{cupper} gives 
\begin{equation}\label{e1lim}
\ee(M) \sim \frac{L\log (1/\beta)}{\beta} \ \ \text{ as } \ \beta \ra 0. 
\end{equation}
Next, the identity \eqref{varm} may be rewritten as
\begin{align*}
\var(M) &= \frac{1}{\beta^2}\sum_{j=1}^\infty \frac{\beta^2 E_j^2e^{-\beta E_j}}{(1-e^{-\beta E_j })^2}\frac{1}{E_j^2}.
\end{align*}
By \eqref{ejbd} the series $\sum 1/E_j^2$ converges. On the other hand, the term
\[
\frac{\beta^2 E_j^2e^{-\beta E_j}}{(1-e^{-\beta E_j })^2}
\]
is uniformly bounded over $\beta$ and $j$ and converges to $1$ as $\beta \ra 0$. Thus, $\var(M)$ behaves asymptotically like $\beta^{-2}\sum 1/E_j^2$ as $\beta \ra 0$. Combined with~\eqref{e1lim}, this completes the proof of the lemma. 
\end{proof}
Define
\[
Z_0 := n - \sum_{j=1}^\infty Z_j = n-M. 
\]
Note that $Z_0$ may be negative. The following simple lemma gives the crucial connection between the $Z_i$'s and the  $N_i$'s. Shortly after our manuscript was posted, Paolo Dai Pra and Francesco Caravenna communicated to us that a version of this lemma is stated and used in their recent probability textbook~\cite{daipra}. Incidentally, the connection between geometric random variables and Bose-Einstein condensation is well known in the probabilistic folklore; it is therefore quite likely that this lemma has been discovered and used by probabilists and mathematical physicists in the past.
\begin{lmm}\label{cond}
Take any $n \ge 1$ and $\beta > 0$. The  canonical ensemble for $n$ particles at inverse temperature $\beta$ is the same as the law of $(Z_0, Z_1,Z_2,\ldots)$ conditioned on the event $M\le n$.
\end{lmm}
\begin{proof}
Let $(N_0, N_1,\ldots)$ be a configuration chosen from the canonical ensemble for $n$ particles at inverse temperature $\beta$. For any sequence of non-negative integers $n_0, n_1,\ldots$ summing to $n$,
\begin{align*}
\pp(N_0 = n_0, N_1 = n_1,\ldots) &= \frac{e^{-\beta \sum_{j=1}^\infty n_j E_j}}{Z_n(\beta)}, 
\end{align*}
where $Z_n(\beta)$ is the normalizing constant. Observe that
\begin{align*}
Z_n(\beta)\prod_{j=1}^\infty(1-e^{-\beta E_j}) &= \biggl(\sum_{n_0, n_1,n_2,\ldots \text{ such that}\atop  n_0+n_1+n_2\cdots = n}e^{-\beta \sum_{j=1}^\infty n_j E_j}\biggr)\prod_{j=1}^\infty(1-e^{-\beta E_j})\\
&=  \biggl(\sum_{n_1,n_2,\ldots \text{ such that}\atop  n_1+n_2+\cdots \le n}e^{-\beta \sum_{j=1}^\infty n_j E_j}\biggr)\prod_{j=1}^\infty(1-e^{-\beta E_j})\\
&= \pp(M\le n). 
\end{align*}
On the other hand
\[
\pp(Z_1=n_1, Z_2=n_2,\ldots) = e^{-\beta \sum_{j=1}^\infty n_j E_j}\prod_{j=1}^\infty(1-e^{-\beta E_j}). 
\]
The proof follows as a consequence of the last three displays. 
\end{proof}

\begin{proof}[Proof of Theorem \ref{wlln}]
First suppose that $\alpha > 1$. Take any $n$ and let $\beta = 1/(k_B T)$, where $T= t n^{1/\alpha}$. Let $Z_0, Z_1,\ldots$ and $M$ be as above. Then by Lemma \ref{wllnlemma}, 
\begin{align*}
\frac{Z_0}{n} &= 1 -\frac{M}{n} \\
&= 1- (k_B t\beta)^\alpha M \\
&\ra 1- (k_Bt)^\alpha L\alpha \Gamma(\alpha) \zeta(\alpha)  = 1-(t/t_c)^\alpha \ \text{ in probability as $n\ra\infty$.}
\end{align*}
By Lemma \ref{cond}, the distribution of $N_0$ is the same as that of $Z_0$ conditional on the event $M\le n$. To finish the proof, note that since $t< t_c$ and by Lemma \ref{wllnlemma} $M/n \ra (t/t_c)^\alpha$, therefore $\pp(M\le n)\ra 1$ as $n \ra\infty$. 

When $\alpha =1$ the proof follows similarly from Lemma \ref{wlln1}, taking $\beta = 1/(k_BT)$, where $T= tn/\log n$. 
\end{proof}
Next we prove Theorem \ref{ellnthm}, since the proof is similar to that of Theorem~\ref{wlln}. Define 
\[
R := \sum_{j=1}^\infty E_j Z_j. 
\]
\begin{lmm}\label{elln}
As $\beta \ra 0$, $\beta^{1+\alpha} R$ converges to $L\alpha \Gamma(\alpha+1)\zeta(\alpha+1)$ in probability. 
\end{lmm}
\begin{proof}
The proof is similar to the proof of Lemma \ref{wllnlemma}. First, note that 
\begin{align}\label{eemr}
\ee(R) &= \sum_{j=1}^\infty \frac{E_j}{e^{\beta E_j } -1}
\end{align}
and 
\begin{align}\label{varmr}
\var(R) = \sum_{j=1}^\infty \frac{e^{-\beta E_j}E_j^2}{(1-e^{-\beta E_j })^2}. 
\end{align}
By Lemma \ref{betalmm} and the above formulas, we get
\begin{align*}
\lim_{\beta\ra 0} \beta^{1+\alpha} \ee(R) &= L\alpha\int_0^\infty \frac{x^{\alpha} }{e^x -1} dx\\
&= L\alpha \sum_{k=1}^\infty \int_0^\infty x^{\alpha} e^{-kx} dx\\
&= L\alpha\Gamma(\alpha+1) \sum_{k=1}^\infty k^{-\alpha-1} = L\alpha \Gamma(\alpha+1) \zeta(\alpha+1),
\end{align*}
and 
\begin{align*}
\lim_{\beta \ra 0} \beta^{2+\alpha}\var(R) &= \lim_{\beta \ra 0} \beta^{\alpha} \sum_{j=1}^\infty \frac{\beta^2 E_j^2 e^{-\beta E_j}}{(1-e^{-\beta E_j})^2} \\
&= L\alpha \int_0^\infty \frac{x^{\alpha+1} e^{-x}}{(1-e^{-x})^2}dx < \infty. 
\end{align*}
This show that $\var(\beta^{1+\alpha} R) = O(\beta^\alpha)$ as $\beta \ra 0$. This completes the proof of the lemma. 
\end{proof}
\begin{proof}[Proof of Theorem \ref{ellnthm}]
Take any $n$ and let $\beta = 1/(k_B T)$, where $T= t n^{1/\alpha}$ if $\alpha > 1$ and $T=tn/\log n$ if $\alpha =1$. By Lemma \ref{cond}, the distribution of $E_{tot}$ is the same as that of $R$ conditional on the event $M\le n$. To finish the proof, note that since $t< t_c$, Theorem \ref{wlln} implies that $\pp(M\le n)\ra 1$ as $n \ra\infty$; then invoke Lemma \ref{elln}.  
\end{proof}

\begin{proof}[Proof of Theorem \ref{clt}]
First, suppose that $1\le\alpha<2$. Fix $\beta > 0$ and let $Z_0, Z_1,\ldots$ and $M$ be as before. For each $k$, let 
\[
M_k := \sum_{j=1}^k Z_j. 
\]
For $\xi \in \rr$, define
\[
\phi_k(\beta, \xi) := \ee(e^{\I \xi \beta(M_k-\ee(M_k))})
\]
where $\I = \sqrt{-1}$, and let
\[
\phi(\beta, \xi) := \ee(e^{\I \xi \beta(M-\ee(M))}). 
\]
Since $M_k \ra M$ and $\ee(M_k)\ra \ee(M)$ as $k\ra\infty$, therefore
\[
\lim_{k\ra\infty }\phi_k(\beta, \xi) = \phi(\beta, \xi). 
\]
An explicit computation gives
\[
\phi_k(\beta, \xi) = \prod_{j=1}^k \ee(e^{\I \xi \beta(Z_j-\ee(Z_j))}) = \prod_{j=1}^k\frac{1-e^{-\beta E_j}}{1-e^{\beta(\I \xi - E_j)}}\exp\biggl( - \frac{\I\xi \beta e^{-\beta E_j}}{1-e^{-\beta E_j}}\biggr). 
\]
Let $\log$ denote the principal branch of the logarithm function. Then the above formula shows that 
\[
\phi_k(\beta, \xi) = \exp\biggl(\sum_{j=1}^k \biggl[\log (1-e^{-\beta E_j}) - \log(1-e^{\beta(\I \xi - E_j)}) - \frac{\I\xi \beta e^{-\beta E_j}}{1-e^{-\beta E_j}}\biggr]\biggr).
\]
By the inequality \eqref{ejbd}, it is easy to see that the series on the right hand side converges absolutely as $k\ra\infty$. Thus,
\begin{align}
&\phi(\beta, \xi) \label{phiform}\\
&= \exp\biggl(\sum_{j=1}^\infty \biggl[\log (1-e^{-\beta E_j}) - \log(1-e^{\beta(\I \xi - E_j)}) - \frac{\I\xi \beta e^{-\beta E_j}}{1-e^{-\beta E_j}}\biggr]\biggr).\nonumber 
\end{align}
Fix $\xi$ and let $a_j(\beta)$ denote the $j$th term in the sum. Then 
for any $j$, 
\begin{align}\label{philim1}
\lim_{\beta \ra 0} a_j(\beta) 
&= \log E_j - \log (E_j - \I \xi) - \frac{\I\xi}{E_j}.
\end{align}
Expanding $a_j(\beta)$ in power series of the logarithm, we get
\begin{align*}
a_j(\beta) &= \sum_{k=1}^\infty \biggl(\frac{e^{k(\I\beta\xi - \beta E_j)}}{k} - \frac{e^{-k \beta E_j}}{k} - \I \xi \beta e^{-k\beta E_j}\biggr)\\
&= \sum_{k=1}^\infty \biggl(\frac{e^{\I k \beta\xi}- 1 -  \I k \beta\xi}{k} \biggr)e^{-k\beta E_j}
\end{align*}
Again, the inequality \eqref{ejbd} guarantees that the power series expansions converge absolutely. Using the inequality 
\begin{equation}\label{calc}
|e^{\I x} - 1-\I x|\le \frac{x^2}{2}
\end{equation}
that holds for all $x\in \rr$,
\begin{align*}
|a_j(\beta)|\le \frac{\beta^2\xi^2}{2}\sum_{k=1}^\infty k e^{-k \beta E_j} = \frac{\beta^2\xi^2 e^{-\beta E_j}}{2(1-e^{-\beta E_j})^2} = \frac{\beta^2\xi^2}{2(e^{\beta E_j/2}-e^{-\beta E_j/2})^2}. 
\end{align*}
For any $x\ge 0$, $e^x-e^{-x} \ge 2x$. Thus for any $j$ and $\beta > 0$,
\begin{align*}
|a_j(\beta)|&\le \frac{\xi^2}{2E_j^2}.   
\end{align*}
Note that by \eqref{ejbd} and the assumption that $\alpha < 2$, 
\begin{align}\label{ej2}
\sum_{j=1}^\infty \frac{1}{E_j^2} \le \frac{1}{K^2} \sum_{j=1}^\infty \frac{1}{j^{2/\alpha}}< \infty. 
\end{align}
Together with \eqref{phiform} and \eqref{philim1}, this shows
\begin{align}\label{characlim}
\lim_{\beta \ra 0} \phi(\beta, \xi) = \exp\biggl(\sum_{j=1}^\infty \biggl[\log E_j- \log(E_j - \I \xi) - \frac{\I\xi}{E_j}\biggr]\biggr),
\end{align}
and also that the series on the right converges absolutely.

Let $X_1, X_2,\ldots$ be i.i.d.\ exponential random variables with mean $1$. For each $k$, define
\[
U_k := \sum_{j=1}^k\frac{X_j-1}{E_j}.
\]
Observe that $\{U_k\}_{k\ge 1}$ is a martingale with respect to the filtration generated by the sequence $\{X_k\}_{k\ge 1}$. Moreover, by \eqref{ej2},
\begin{align*}
\ee(U_k^2) &= \sum_{j=1}^k \frac{1}{E_j^2} \le\sum_{j=1}^\infty \frac{1}{E_j^2} < \infty.
\end{align*}
Thus, $U_k$ is a uniformly $L^2$ bounded martingale. Therefore the limit
\[
U := \sum_{j=1}^\infty \frac{X_j-1}{E_j} 
\]
exists almost surely and $U_k \ra U$ in $L^2$. A simple application of the dominated convergence theorem shows that the the  characteristic function of $U$ is given by the formula
\[
\ee(e^{\I \xi U}) = \exp\biggl(\sum_{j=1}^\infty \biggl[\log E_j - \log(E_j - \I \xi) - \frac{\I\xi}{E_j}\biggr]\biggr),
\]
which is exactly the right hand side of \eqref{characlim}. This shows that as $\beta \ra 0$, $\beta(M-\ee(M))$ converges in distribution to $U$. 

Now suppose that $1<\alpha < 2$ and let $\beta = 1/(k_BT)$, where $T=tn^{1/\alpha}$. Since $M-\ee(M) = \ee(Z_0)-Z_0$, the above conclusion may be restated as:
\begin{align*}
\frac{Z_0-\ee(Z_0)}{n^{1/\alpha}} &\stackrel{d}{\longrightarrow} -k_B t U \\
&= k_B t (L\alpha \Gamma(\alpha)\zeta(\alpha))^{1/\alpha} W = (t/t_c)W \ \ \text{ as $n\ra\infty$}. 
\end{align*}
To finish, recall that by Lemma \ref{cond}, the law of $N_0$ is the same as that of $Z_0$ conditional on the event $M\le n$, and by Theorem \ref{wlln}, $\pp(M\le n) \ra 1$ as $n\ra\infty$ when $t< t_c$.

Similarly if $\alpha =1$ let $\beta=1/(k_B T)$ where $T=tn/\log n$. Then as above, 
\[
\frac{N_0-\ee(N_0)}{n/\log n} \stackrel{d}{\longrightarrow}(t/t_c)W. 
\]
Next, consider the case $\alpha > 2$. Fix $\beta > 0$ and let $Z_0, Z_1,\ldots$ and $M$ be as before.  Let
\[
\phi(\beta, \xi) := \ee(e^{\I \xi \beta^{\alpha/2}(M-\ee(M))}). 
\]
A similar computation as before gives
\begin{align*}
\phi(\beta, \xi)
&= \exp\biggl(\sum_{j=1}^\infty \biggl[\log (1-e^{-\beta E_j}) - \log(1-e^{\I\beta^{\alpha/2}\xi - \beta E_j}) - \frac{\I\xi \beta^{\alpha/2} e^{-\beta E_j}}{1-e^{-\beta E_j}}\biggr]\biggr). 
\end{align*}
Expanding in power series gives 
\begin{align*}
\phi(\beta, \xi) &= \exp\biggl(\sum_{j=1}^\infty \sum_{k=1}^\infty \biggl(\frac{e^{k(\I\beta^{\alpha/2}\xi - \beta E_j)}}{k} - \frac{e^{-k \beta E_j}}{k} - \I \xi \beta^{\alpha/2} e^{-k\beta E_j}\biggr)\biggr)\\
&= \exp\biggl(\sum_{j=1}^\infty \sum_{k=1}^\infty \biggl(\frac{e^{\I k \beta^{\alpha/2}\xi}- 1 -  \I k \beta^{\alpha/2}\xi}{k} \biggr)e^{-k\beta E_j}\biggr).
\end{align*}
Note that the double series is absolutely convergent by the inequality~\eqref{ejbd}, because the term within the parenthesis may be bounded by $3\beta^{\alpha/2} \xi$ in absolute value, and   
\begin{align*}
\sum_{j=1}^\infty \sum_{k=1}^\infty e^{-k\beta E_j} &= \sum_{j=1}^\infty \frac{1}{e^{\beta E_j}-1} \le \sum_{j=1}^\infty \frac{1}{e^{\beta K j^{1/\alpha}} - 1} < \infty. 
\end{align*}
Therefore the order of summation may be interchanged. For each $k$, let 
\[
a_k(\beta) := \beta^{-\alpha} \biggl(\frac{e^{\I k \beta^{\alpha/2}\xi}- 1 -  \I k \beta^{\alpha/2}\xi}{k}\biggr)
\]
and 
\[
b_k(\beta) := \beta^\alpha \sum_{j=1}^\infty e^{-k \beta E_j},
\]
so that 
\[
\phi(\beta, \xi) =\exp\biggl( \sum_{k=1}^\infty a_k(\beta) b_k(\beta)\biggr). 
\]
Lemma \ref{betalmm} implies that 
\begin{equation}\label{bklim}
\lim_{\beta \ra 0} b_k(\beta)  = L \alpha \int_0^\infty x^{\alpha -1 }e^{-kx}dx = L\alpha \Gamma(\alpha) k^{-\alpha}. 
\end{equation}
On the other hand,
\begin{equation}\label{aklim}
\lim_{\beta \ra 0} a_k(\beta) = - \frac{k\xi^2}{2}. 
\end{equation}
By \eqref{ejbd}, 
\begin{align}
b_k(\beta)&\le \beta^\alpha \sum_{j=1}^\infty e^{-k\beta Kj^{1/\alpha}}\le \beta^\alpha \int_0^\infty e^{-k\beta K x^{1/\alpha}} dx\nonumber \\
&= k^{-\alpha} \int_0^\infty e^{-K y^{1/\alpha}} dy \label{bkbd}
\end{align}
and by \eqref{calc}, 
\begin{align*}
|a_k(\beta)| &\le \frac{k\xi^2}{2}. 
\end{align*}
Consequently, 
\[
|a_k(\beta)b_k(\beta)|\le C k^{1-\alpha},
\]
where $C$ is a constant that does not depend on $k$ or $\beta$. 
Since $\alpha > 2$, therefore by \eqref{bklim} and \eqref{aklim} and the dominated convergence theorem, 
\begin{align*}
\lim_{\beta \ra 0} \phi(\beta, \xi) &= \exp\biggl(-\frac{1}{2}L\alpha \Gamma(\alpha)\xi^2\sum_{k=1}^\infty k^{1-\alpha}\biggr) \\
&=\exp\biggl( -\frac{1}{2}L\alpha \Gamma(\alpha)\zeta(\alpha-1)\xi^2\biggr). 
\end{align*}
This shows that 
\[
\beta^{\alpha/2}(M-\ee(M)) \stackrel{d}{\longrightarrow} {\mathcal N}(0, L\alpha \Gamma(\alpha)\zeta(\alpha-1)) \ \ \text{ as $\beta \ra 0$}. 
\]
Now let $\beta = 1/(k_BT)$, where $T=t n^{1/\alpha}$. Since $M-\ee(M) = \ee(Z_0)-Z_0$, 
\begin{align*}
\frac{Z_0-\ee(Z_0)}{\sqrt{n}} &\stackrel{d}{\longrightarrow} {\mathcal N}(0,(k_B t)^{\alpha} L\alpha \Gamma(\alpha)\zeta(\alpha-1)) \\
&= {\mathcal N}(0, (t/t_c)^\alpha \zeta(\alpha-1)/\zeta(\alpha)). 
\end{align*}
As in the case $1< \alpha < 2$, Theorem \ref{wlln} and Lemma \ref{cond} allow replacing $Z_0$ by $N_0$ in the above display.

Finally, suppose that $\alpha =2$. Fix $\beta\in (0,1)$ and let 
\[
\gamma = \gamma(\beta) := \frac{\beta}{\sqrt{\log (1/\beta)}}\ . 
\]
Let $Z_0, Z_1,\ldots$ and $M$ be as before, and let
\[
\phi(\beta, \xi) := \log \ee(e^{\I \xi \gamma(M-\ee(M))}). 
\]
As in the case $\alpha >2$, this gives the formula
\begin{align*}
\phi(\beta, \xi) &= \exp\biggl(\sum_{j=1}^\infty \sum_{k=1}^\infty \biggl(\frac{e^{\I k \gamma\xi}- 1 -  \I k \gamma\xi}{k} \biggr)e^{-k\beta E_j}\biggr).
\end{align*}
Using \eqref{ejbd}, the double series is absolutely convergent. Fix $\xi$, and define 
\[
a_k(\beta) := \beta^{-2} \biggl(\frac{e^{\I k \gamma\xi}- 1 -  \I k \gamma\xi}{k}\biggr)
\]
and 
\[
b_k(\beta) := \beta^2 \sum_{j=1}^\infty e^{-k \beta E_j},
\]
so that 
\[
\phi(\beta, \xi) =\exp\biggl( \sum_{k=1}^\infty a_k(\beta) b_k(\beta)\biggr). 
\]
By the inequality
\[
|e^{\I x} - 1 - \I x|\le 2|x|
\] 
we have
\[
|a_k(\beta)|\le \frac{2|\xi|}{\beta \sqrt{\log (1/\beta)}},
\]
and by the same argument that led to \eqref{bkbd}, 
\begin{equation}\label{bkbd2}
b_k(\beta) \le Ck^{-2},
\end{equation}
where 
\[
C = \int_0^\infty e^{-Ky^{1/2}} dy. 
\]
Thus, taking
\[
\eta = \eta(\beta) := \frac{1}{\beta (\log (1/\beta))^{1/4}}\ ,
\]
the last two inequalities give 
\begin{align*}
\sum_{k > \eta} |a_k(\beta) b_k(\beta)|\le \frac{2C|\xi|}{\beta \sqrt{\log(1/\beta)}} \sum_{k> \eta} k^{-2}. 
\end{align*}
Since the right hand side tends to zero as $\beta \ra 0$, we see that 
\begin{equation}\label{remain1}
\lim_{\beta \ra 0} \frac{\phi(\beta,\xi)}{\phi_1(\beta,\xi)} = 1,
\end{equation}
where 
\[
\phi_1(\beta,\xi) := \exp\biggl(\sum_{k\le \eta} a_k(\beta)b_k(\beta)\biggr). 
\]
Next, note that the inequality
\[
\biggl|e^{\I x} - 1- \I x + \frac{x^2}{2}\biggr| \le \frac{|x|^3}{6}
\]
gives
\begin{align}\label{akbkbd}
\biggl|a_k(\beta) + \frac{k\xi^2}{2\log(1/\beta)}\biggr| &\le \frac{k^2\beta |\xi|^3}{12(\log(1/\beta))^{3/2}}. 
\end{align}
By \eqref{bkbd2} and \eqref{akbkbd}, 
\begin{align*}
\biggl|\sum_{k\le \eta}a_k(\beta) b_k(\beta) - \frac{\xi^2}{2\log(1/\beta)}\sum_{k\le \eta}kb_k(\beta) \biggr| &\le \frac{\beta |\xi|^3}{12(\log(1/\beta))^{3/2}}\sum_{k\le \eta} k^2b_k(\beta)\\
&\le \frac{C|\xi|^3}{12(\log(1/\beta))^{7/4}}. 
\end{align*}
Thus, defining
\[
\phi_2(\beta, \xi) := \exp\biggl(-\frac{\xi^2}{2\log(1/\beta)}\sum_{k\le \eta} k b_k(\beta)\biggr) 
\]
gives
\begin{equation}\label{remain2}
\lim_{\beta \ra 0} \frac{\phi_1(\beta,\xi)}{\phi_2(\beta,\xi)} = 1. 
\end{equation}
Define
\[
\delta(\beta) := \max_{k\le \eta} |k^2b_k(\beta)- 2L|.
\]
We claim that 
\begin{equation}\label{deltalim}
\lim_{\beta \ra 0} \delta(\beta) = 0.
\end{equation}
To see this, suppose not. Then there exists a sequence $\beta_i \ra 0$ and a sequence of integers $k_i$ such that $k_i \le\eta_i:= \eta(\beta_i)$ for all $i$, and 
\[
\lim_{i\ra \infty} k_i^2b_{k_i}(\beta_i)\ne 2L. 
\]
However, since $\beta_i k_i \ra 0$, Lemma \ref{betalmm} implies the above limit must be equal to $2L$. This gives a contradiction which proves \eqref{deltalim}. Now note that 
\begin{align*}
&\biggl|\frac{1}{2\log(1/\beta)}\sum_{k\le \eta} kb_k(\beta) -  \frac{L}{\log(1/\beta)} 
\sum_{k\le \eta} \frac{1}{k}\biggr|\\
&\le \frac{1}{2\log(1/\beta)}\sum_{k\le \eta}\frac{\delta(\beta)}{k}\le \frac{\delta(\beta)\log\eta}{2\log(1/\beta)}.
\end{align*}
From \eqref{deltalim}, the above bound tends to zero as $\beta \ra 0$. 
Thus,
\[
\lim_{\beta \ra 0} \phi_2(\beta, \xi) = e^{ -L \xi^2}. 
\]
Therefore by \eqref{remain1} and \eqref{remain2}, $\phi(\beta, \xi)$ also tend to the same limit as $\beta \ra 0$. In other words,
\[
\frac{\beta(M-\ee(M))}{\sqrt{\log(1/\beta)}} \stackrel{d}{\longrightarrow} {\mathcal N}(0, 2L) \ \ \text{ as $\beta \ra 0$.}
\]
Now let $\beta = 1/(k_BT)$, where $T=t\sqrt{n}$. Then, as $n\ra\infty$, the above display may be written as
\[
\frac{M-\ee(M)}{\sqrt{n\log n}} \stackrel{d}{\longrightarrow} {\mathcal N}(0, k_B^2 t^2L) = {\mathcal N}(0,(t/t_c)^2/(2\zeta(2))).
\]
Proceeding as in the other two cases, we get the central limit theorem for~$N_0$. This completes the proof of Theorem \ref{clt}. 
\end{proof}
\begin{proof}[Proof of Theorem \ref{other}]
First suppose that $\alpha > 1$. Recall that for $j\ge 1$, $Z_j$ is a geometric random variable with mean $1/(e^{\beta E_j} -1 )$. Consequently, the limiting distribution of $\beta Z_j$ as $\beta \ra 0$ is exponential with mean $1/E_j$. Taking $\beta = 1/(k_B T)$ where $T=tn^{1/\alpha}$, we get that $Z_j/n^{1/\alpha}$ converges in law to exponential with mean $k_B t/E_j$. By Lemma \ref{cond}, the distribution of $N_j$ is the same as that of $Z_j$ conditional on the event $M\le n$. Since $t< t_c$, Lemma \ref{wllnlemma} implies that $\pp(M\le n) \ra 1$ as $n\ra\infty$. Thus $Z_j/n^{1/\alpha}$ and $N_j/n^{1/\alpha}$ have the same limiting distribution. 

Note that the above argument was carried out under the assumption that $E_0=0$. If $E_0\ne 0$, simply replace $E_j$ by $E_j-E_0$. 

When $\alpha=1$, the same argument can be carried out with $\beta = 1/k_B T$, where $T=tn/\log n$. 
\end{proof}

\begin{proof}[Proof of Theorem \ref{wtail}]
We are aided in this classical computation by the work of Montgomery and Odlyzko \cite{mo}. They proved similar theorems for Bernoulli random variables. Their paper contains further refinements which could also be developed for the present example.

Without loss of generality, assume that $E_0=0$. The random variable $(1-X_j)/E_j$ has mean zero and moment generating function
\[
\ee(e^{\lambda(1-X_j)/E_j}) = \frac{e^{\lambda/E_j}}{1+\lambda/E_j}. 
\]
The following bounds are needed: For any $x\ge 0$,
\begin{align}
&\frac{e^x}{1+x} \le e^x, \label{p1}\\
&\frac{e^x}{1+x} \le e^{x^2/2}. \label{p2}
\end{align}
Further,
\begin{align}
&\frac{e^x}{1+x} \ge e^{x/2}\ \text{ for } \ x\ge 3,\label{p3}\\
&\frac{e^x}{1+x} \ge e^{x^2/6} \ \text{ for } \ 0\le x\le 3. \label{p4}
\end{align}
The bound \eqref{p1} is obvious. For \eqref{p2}, consider two cases: (a) $0< x< 1$; then $e^x < (1+x)e^{x^2/2}$ is equivalent to
\[
x < \frac{x^2}{2} + x - \frac{x^2}{2} + \frac{x^3}{3} -\cdots,
\]
or equivalently
\[
0  < \biggl(\frac{x^3}{3} - \frac{x^4}{4}\biggr) + \biggl(\frac{x^5}{5} - \frac{x^6}{6}\biggr) + \cdots.
\]
This last inequality is true termwise. 

(b) $1\le x< \infty$; here 
\[
\frac{e^x}{1+x} \le \frac{e^x}{2} < e^{x^2/2}
\]
if and only if 
\[
x< \frac{x^2}{2}+\log 2,
\]
but $x^2/2 - x+\log 2$ has a unique minimum at $1$ where it is positive. This completes the proof of \eqref{p2}. 

For \eqref{p3}, note that $e^x/(1+x) > e^{x/2}$ is equivalent to $e^{x/2} > 1+x$, which may be written as
\[
\frac{x^2}{8} + \frac{x^3}{48} + \cdots > \frac{x}{2}. 
\]
This inequality is true for $x\ge 3$. 

For \eqref{p4}, again work in cases. For $0\le x < 1$, taking logs allows checking 
\[
x > \frac{x^2}{6} + x - \frac{x^2}{2} + \frac{x^3}{3} - \cdots
\]
or equivalently
\[
\biggl( -\frac{x^2}{3}+\frac{x^3}{3} \biggr) + \biggl(-\frac{x^4}{4} + \frac{x^5}{5}\biggr) + \cdots < 0. 
\]
The left side has each term negative. A similar check, expanding the logarithm and pairing terms, works for $1\le x\le 2$ and $2\le x\le 3$. 

The stated upper bounds in Theorem \ref{wtail} now follow from the argument of Montgomery and Odlyzko word for word with \eqref{p1}--\eqref{p4} substituting for their inequalities (7)--(10). We omit further details. 

Recall the definition \eqref{sdef} of $S(\lambda)$. Define 
\[
S(\lambda-):= \lim_{x\uparrow \lambda} S(x). 
\]
When $E_j$'s satisfy \eqref{weyl}, $S(\lambda)$ and $S(\lambda-)$ are both asymptotic to $L\lambda^\alpha$. Since $S(E_j) \ge j$ and $S(E_j-)\le j$ and $E_1> 0$, there are positive constants $a$ and $b$ such that
\[
a j^{1/\alpha} \le E_j \le b j^{1/\alpha}
\]
for each $j\ge 1$. Consequently, if $\alpha>1$, $n_x$ and $n_x'$ are bounded above and below by constant multiples of $x^{\alpha/(\alpha-1)}$. This, combined  with \eqref{wx1} and \eqref{wx2}, implies \eqref{a1a2}. Similarly for \eqref{a1a22}, note that if $\alpha=1$ then $\log n_x$ and $\log n_x'$ are both bounded above and below by constant multiples of $x$.

Finally, to prove \eqref{wx3}, notice that the random variable $-W$ has moment generating function
\[
\phi(\lambda) = \ee(e^{-\lambda W}) = \prod_{j=1}^\infty \frac{e^{-\lambda/E_j}}{1-\lambda/E_j}, \ \ -\infty < \lambda < E_1.
\]
By Markov's inequality, 
\[
\pp(W\le -x) \le e^{-\lambda x} \phi(\lambda) 
\]
for any $0< \lambda < E_1$. This proves \eqref{wx3}. To see that no power $x^{1+\ep}$ will do, observe that $\phi(E_1)=\infty$. A tail bound with $x^{1+\ep}$ would give $\phi(\lambda)< \infty$ for all $\lambda$. 
\end{proof}

\begin{proof}[Proof of Proposition \ref{gumbel}]
Take any $k$ and let $Y_1, Y_2,\ldots, Y_k$ are i.i.d.\ exponential random variables with mean $1$. Let $Y_{(j)}$ be the $j$th largest $Y_i$, and $Y_{(k+1)}=0$. Let $Z_j:= Y_{(j)}-Y_{(j+1)}$. It is a simple fact that the random variables $Z_1, \ldots, Z_k$ are independent, and $Z_j$ is exponentially distributed with mean $1/j$. Thus, if $X_1,X_2,\ldots$ are i.i.d.\ exponential random variables with mean $1$, then the distribution of $\sum_{j=1}^k X_j/j$  is the same as that of $\sum_{j=1}^k Z_j$. But note that $\sum_{j=1}^k Z_j = Y_{(1)}$. Thus, if 
\begin{align*}
W_k := \sum_{j=1}^k \frac{1-X_j}{j},
\end{align*}
then for all $-\infty <x\le \sum_{j=1}^k 1/j$, 
\begin{align*}
\pp(W_k \ge x) &= \pp\biggl(Y_{(1)} \le \sum_{j=1}^k \frac{1}{j} -x\biggr) \\
&= (1-e^{x-\sum_{j=1}^k 1/j})^k.
\end{align*}
Thus,
\begin{align*}
\pp(W\ge x) &= \lim_{k\to \infty}(1-e^{x-\sum_{j=1}^k 1/j})^k\\
&= \exp\biggl(-\lim_{k\to \infty}e^{x-\sum_{j=1}^k 1/j + \log k}\biggr)\\
&= e^{-e^{x-\gamma}}. 
\end{align*}
This completes the proof of the proposition. Incidentally, a version of this proof may also be found in \cite[Example B.11]{janson}.
\end{proof}

\vskip.2in
\noindent{\bf Acknowledgments.} The authors would like to gratefully acknowledge James Zhao, Megan Bernstein and Susan Holmes for their  generous help with the computer; Elliott Lieb, Werner Krauth, Kannan Soundararajan, Joe Pul\'e, Valentin Zagrebnov, Paolo Dai Pra, Francesco Carvenna and Charles Radin for their comments and suggestions; and Lenny Susskind and Steve Shenker for their inspiring course on quantum mechanics at Stanford that was instrumental in getting the authors interested in this topic. The authors particularly thank the referees for pointing out many references that were missed in the first draft.

\end{document}